\documentclass[12pt,reqno]{amsart}
\usepackage{amsmath,amssymb,amsthm,amsfonts,graphicx,color,xcolor}
\usepackage{bm}
\usepackage{hyperref}
\usepackage{textcomp}
\usepackage{yfonts}
\usepackage{amssymb}
\usepackage[hmargin=3.2cm, vmargin=3.2cm]{geometry}
\usepackage{dsfont}
\usepackage{upgreek} 
\usepackage{mathrsfs}
\usepackage{cancel}
\usepackage{framed}
\usepackage{pgfplots}
\usepackage{ulem}

\usepackage{pgfplots}
\pgfplotsset{compat=1.18}
\usepackage{siunitx} 

\usepackage{orcidlink}

\usepackage{tikz}
\usetikzlibrary{positioning}

\numberwithin{equation}{section}
\newtheorem{conjecture}{Conjecture}[section]

\newtheorem{thm}{Theorem}[section]
\newtheorem{lemma}{Lemma}[section]
\newtheorem{definition}{Definition}[section]

\newtheorem{proposition}{Proposition}[section]

\hypersetup{colorlinks,linkcolor={blue},citecolor={blue},urlcolor={blue}}
\newcommand{\R}{{\mathbb R}}

\newcommand{\N}{{\mathbb N}}
\newcommand{\Z}{{\mathbb Z}}

\newcommand{\kn}{\left(k,\frac{2}{n}\right)} 

\renewcommand{\j}{{\bm j}} 

\begin{document}  
	\title[]{A Schwartz-type Space for the $\left(k,\frac{2}{n}\right)-$Generalized Fourier Transform}
	
	\author{{{Nelson Faustino}~\orcidlink{0000-0002-9117-2021}
	}}

	\address[\textsc{Nelson~Faustino}]{Department of Mathematics and Center for R\&D in Mathematics and Applications (CIDMA), University of Aveiro, Campus Universit\'ario de Santiago, 3810-193 Aveiro, Portugal}
	\email{\href{mailto:nfaust@ua.pt}{nfaust@ua.pt}}
	\thanks{{Nelson~Faustino}~is supported by CIDMA under the Portuguese Foundation for Science and Technology 
	(FCT, \href{https://ror.org/00snfqn58}{https://ror.org/00snfqn58})  
	Multi-Annual Financing Program for R\&D Units,
	grants UID/4106/2025 and UID/PRR/4106/2025.}
	
	\author{{Selma Negzaoui}~\orcidlink{0000-0003-0898-1256}
	}
	\address[\textsc{Selma Negzaoui}]{Preparatory Institute of Engineering Studies of Monastir, University of Monastir, and Universit\'e de Tunis El Manar, Facult\'e des Sciences de Tunis. Laboratoire d'Analyse Math\'ematique et Applications LR11ES11, 2092 Tunis, Tunisie.
	}
	\email{\href{mailto:selma.negzaoui@fst.utm.tn}{selma.negzaoui@fst.utm.tn}}
   
\begin{abstract}
The Schwartz space $\mathcal{S}(\mathbb{R}^N)$ is not invariant under the $(k,a)$-generalized Fourier transform $\mathcal{F}_{k,a}$ unless $a=2$, and in general no such adapted space is known. For $N=1$ and $\displaystyle a=\frac{2}{n}$, $n\in\mathbb{N}$, we construct a tailored Schwartz-type space $\mathcal{S}_{k,n}(\mathbb{R})$ defined via seminorms built from natural second-order operators associated with the one-dimensional Dunkl Laplacian $\Delta_k$. We prove that  $\mathcal{S}_{k,n}(\mathbb{R})$ recovers the two basic features of the classical Schwartz space: invariance under the corresponding Fourier-type operator and density in the relevant weighted $L^p-$spaces.
To establish these results, we introduce the space $\mathcal{D}_{k,n}(\mathbb{R})$ of compactly supported smooth functions, which embeds continuously into $\mathcal{S}_{k,n}(\mathbb{R})$ and is dense in the weighted spaces $L^p(d\mu_{k,n})$, $1\le p<\infty$. These results provide the first Schwartz-type space for $\mathcal{F}_{k,a}$ that simultaneously ensures invariance and $L^p$-density, and admits an $\mathfrak{sl}(2,\mathbb{R})$-based description of the underlying operator structure.
\end{abstract}



\maketitle


\section{Introduction}

\subsection{State of the Art}

The Schwartz space $\mathcal{S}(\mathbb{R}^N)$ provides a natural framework for Fourier analysis and tempered distributions because it is invariant under the Fourier transform $\mathcal{F}$. It is also dense in all $L^p-$spaces. Both the Hankel transform~\cite{Tr2001Book} and the Dunkl transform~\cite{Jeu93,ChTr,Tr02} share this invariance (we refer, for instance, to~\cite[Chapter VII]{Hormander03} for a classical discussion). 
However, for more general Fourier-like operators, such as the Jacobi and the Chebli-Trim\`eche operators among others (cf.~\cite{De,BeDa,Tr81}),  this picture breaks down.  \(\mathcal{S}(\mathbb R^N)\) typically fails to remain invariant, and no canonical replacement is available for the $(k,a)$-generalized Fourier transform (cf.~\cite{GoIvTi23}).

The $(k,a)$-generalized Fourier transform $\mathcal{F}_{k,a}$, introduced by Ben Sa\"id, Kobayashi and \O rsted \cite{BKO} by considering the $(k,a)$-deformation of the Hermite operator $\displaystyle \Delta-~\|x\|^2$, induced by the Laguerre-type operator
\begin{eqnarray*}
	\Delta_{k,a}:=\| x\|^{2-a}\Delta_{k}-\| x\|^a,&a>0,
\end{eqnarray*}
where $\Delta_{k}$ is the Dunkl Laplacian 
studied by Dunkl \cite{Du89,Du91,Du92}, de Jeu \cite{Jeu93,Jeu06}, R\"osler \cite{Ro98,Ro99,Ro03,Ro03-book}, Trim\`eche \cite{Tr01,Tr02} among others. The associated $(k,a)-${\it generalized Fourier transform} is introduced as follows:
\begin{eqnarray}	\label{FkaRN}	\mathcal{F}_{k,a}=\exp\left(\frac{i\pi}{2a}\left(~2\langle k \rangle+N+a-2~ \right)\right)\exp\left(\frac{i\pi}{2a}\Delta_{k,a}\right),\end{eqnarray}
where $k := \langle k \rangle$ denotes a multiplicity function invariant under the reflection group action.

The $(k,a)$-generalized Fourier transform admits a full spectral and kernel theory, including integral representations. However, no Schwartz space adapted to $\mathcal{F}_{k,a}$ was constructed. 
Recently, Gorbachev, Ivanov and Tikhonov demonstrated in \cite{GoIvTi23} that the classical framework is in fact inadequate in most cases. They proved that 
\[ \mathcal{F}_{k,a}\bigl(\mathcal{S}(\R^N)\bigr)\not\subset \mathcal{S}(\R^N) \] 
except when \(a=2\) (Dunkl case), by constructing an explicit counterexample based on the Gaussian function (see~\cite[Example 5.1]{GoIvTi23}). 

In this paper, we rigorously treat the case \(N=1\) and \(\displaystyle a=\frac{2}{n}\), where $n\in \N$. 
In this setting, we introduce a Schwartz-type space $\mathcal{S}_{k,n}(\mathbb{R})$, adapted to the $\kn-$generalized Fourier transform $\mathcal{F}_{k,n}:=\mathcal{F}_{k,\frac{2}{n}}$. First, we prove that $\mathcal{S}_{k,n}(\mathbb{R})$ is invariant under the transform $\mathcal{F}_{k,n}$. More precisely,
$$
\mathcal{F}_{k,n}\big(\mathcal{S}_{k,n}(\mathbb{R})\big)=\mathcal{S}_{k,n}(\mathbb{R}).
$$

This enables the mapping property \(\mathcal{F}_{k,n}^m: S_{k,n}(\mathbb{R}) \to S_{k,n}(\mathbb{R})\) for all $m\in \N$, as in the classical case.

Second, by considering the weighted Lebesgue spaces $L^p(d\mu_{k,n}):= L^p(\R,d\mu_{k,n})$, where the measure $d\mu_{k, n}$ is given by
\begin{equation}
	\label{knMeasure}
	d\mu_{k ,n}(x)=c_{k ,n}|x|^{2k +\frac{2}{n}-2}dx, \quad\text{and}\quad c_{k ,n}=\frac{1}{2\Gamma(k  n-\frac{n}{2}+1)}\left(\frac{n}{2}\right)^{k  n-\frac{n}{2}},\end{equation} 
we establish that 
$
\mathcal{S}_{k,n}(\mathbb{R}) \subset L^p(d\mu_{k,n}), \; 1\le p<\infty,
$
with dense embedding. So functions in $L^p(d\mu_{k,n})$ can be approximated in the $L^p$-norm by very ``regular'' functions from our space $\mathcal{S}_{k,n}(\mathbb{R})$. 

We also introduce a $\kn-$analogue $\mathcal{D}_{k,n}(\mathbb{R})$ of the space $\mathcal{D}(\mathbb{R})$ of compactly supported smooth functions, and we prove that \[ D_{k,n}(\mathbb{R}) \hookrightarrow S_{k,n}(\mathbb{R}) \hookrightarrow L^p(d\mu_{k,n}), \qquad 1\le p<\infty. \]

Our construction of the Schwartz type space is based on a Lie-algebraic
characterization via the $\mathfrak{sl}(2,\mathbb{R})$ generators.
More precisely, $\mathcal{S}_{k,n}(\mathbb{R})$ is defined by replacing $X=xI$ and $\displaystyle D=\frac{d}{dx}$ in the seminorms of $\mathcal{S}(\mathbb{R})$ by the operators $\left|~x~\right|^{\frac{2}{n}}I$ and $\left|~x~\right|^{2-\frac{2}{n}}\Delta_k$, and the function $f$ by the sequence of functions $\left(X^\ell D^\ell f\right)_{\ell\in\mathbb{N}_0}$.

\subsection{Organization of the paper}
In Section~\ref{ApproachSection}, we state the three main theorems of the paper. In Section~\ref{Preliminaries} we recall the relevant properties of the \(\kn-\)generalized kernel and the associated translation and convolution operators. In Section~\ref{knSchwartz-section} we develop the \(\mathfrak{sl}(2,\mathbb R)-\)based description of \(\mathcal{S}_{k,n}(\mathbb R)\) and prove the invariance of \(\mathcal{S}_{k,n}(\mathbb R)\) under the $\kn-$generalized Fourier transform, as well as the continuity of the embedding \(\mathcal{D}_{k,n}(\mathbb R)\hookrightarrow \mathcal{S}_{k,n}(\mathbb R)\). In Section~\ref{Lp-section}, we prove the embedding and density properties in \(L^p(d\mu_{k,n})\). Finally, in Section~\ref{ConclusionOpen} we discuss how our one-dimensional construction serves as a test case for higher dimensions and more general parameters \(a\), and formulates a conjecture on the existence of adapted Schwartz-type spaces \(\mathcal{S}_{k,a}(\mathbb R^N)\), in a broader context.

\section{Definitions and Main Results}\label{ApproachSection}

Here and throughout the paper, we use the standard notations
\begin{eqnarray*}
	\displaystyle \N=\{n\in \Z~:~n\geq 1\}, \text{ and }&\N_0=\N\cup \{0\} .
\end{eqnarray*}

We observe that the Schwartz space $\mathcal{S}(\mathbb{R})$ consists of functions $f\in C^\infty(\mathbb{R})$ satisfying the seminorm condition
\begin{eqnarray} 
	\label{sup-Schwartz} \sup_{x\in \R} \left|~X^\alpha D^\beta f(x) ~\right|<\infty, & \text{for all} & \alpha,\beta \in \N_0. 
\end{eqnarray}

Here, $\displaystyle X = xI$, $\displaystyle D = \frac{d}{dx}$ and the identity operator $I$ induce the Weyl--Heisenberg algebra of dimension 3, whose Lie algebraic structure underlies the invariance of $\mathcal{S}(\mathbb{R})$ for the Fourier transform. 

In our setting, however, the Weyl–Heisenberg picture breaks down, because there is no $\kn-$generalized Dunkl operator $T_{k,\frac{2}{n}}$ satisfying $\left(T_{k,\frac{2}{n}}\right)^2=\left|x\right|^{2-\frac{2}{n}}\Delta_{k}$ for $\frac{2}{n}\neq 2$. 
Note that \(\Delta_k\), in the one-dimensional setting, is a differential-difference operator associated with the abelian group \(\mathbb{Z}_2\). It acts componentwise as:
\begin{eqnarray}\label{DunklLaplacian}
	(\Delta_{k} f)(x)=f''(x)+\frac{2k}{x}f'(x)-k\frac{f(x)-f(-x)}{x^2}, & x \in \R \setminus \{0\}.
\end{eqnarray}

Following \cite{BKO}, we replace the Weyl--Heisenberg generators $X,D$ by $\kn$-analogues that realize $\mathfrak{sl}(2,\R)$, yielding $\mathcal{S}_{k,n}(\mathbb{R})$:
\begin{definition}\label{knSchwartz}
	We define $\mathcal{S}_{k,n}(\R)$ as the space of all {$f\in C^\infty(\R\setminus\{0\})$}
	such that 
	\begin{eqnarray}
		\label{sup-kn}
		\sup_{x\in \R\setminus\{0\}}~\left|~\left(\left|~x~\right|^{\frac{2}{n}}\right)^{\alpha} \left(\left|~x~\right|^{2-\frac{2}{n}}\Delta_{k}\right)^\beta \left(x^\ell f^{(\ell)}(x)\right)~\right|<\infty, &\mbox{for all}& \alpha,\beta,\ell \in \N_0.
	\end{eqnarray}
\end{definition}

In short, $\mathcal{S}_{k,n}(\R)$ is defined from $\mathcal{S}(\R)$ by replacing $X,D,f$ with $\left|~x~\right|^{\frac{2}{n}}I$, $\left|~x~\right|^{2-\frac{2}{n}}\Delta_{k}$, $(X^\ell D^\ell f)_{\ell\in\N_0}$, respectively. This parallels the generators of $\mathfrak{sl}(2,\R)$, defined as (cf. \cite[Theorem 3.2]{BKO}):
\begin{eqnarray}
	\label{LadderOp}
	\mathbb{E}_{k,\frac{2}{n}}^+=\dfrac{in}{2}\left|~x~\right|^{\frac{2}{n}},~~\mathbb{E}_{k,\frac{2}{n}}^-=\dfrac{in}{2}\left|~x~\right|^{2-\frac{2}{n}}\Delta_k,~~
	\mathbb{H}_{k,\frac{2}{n}}=nx\frac{d}{dx}+\left(kn+1-\frac{n}{2}\right)I,
\end{eqnarray} 

These generators, satisfying the commutation relations
\begin{eqnarray}
	\label{sl2} \left[\mathbb{E}_{k,\frac{2}{n}}^+,\mathbb{E}_{k,\frac{2}{n}}^-\right]=\mathbb{H}_{k,\frac{2}{n}}, & \left[\mathbb{H}_{k,\frac{2}{n}},\mathbb{E}_{k,\frac{2}{n}}^+\right]=2\mathbb{E}_{k,\frac{2}{n}}^+, & \left[\mathbb{H}_{k,\frac{2}{n}},\mathbb{E}_{k,\frac{2}{n}}^-\right]=-2\mathbb{E}_{k,\frac{2}{n}}^-,
\end{eqnarray}
are encoded in the one‑dimensional intertwining relations established in~\cite[Theorem 5.6]{BKO} for a dense subspace of $L^2(d\mu_{k,n})$:
\begin{eqnarray}
	\label{knFourierIntertwining}
	\begin{array}{lll}
		\displaystyle \mathcal{F}_{k,n}\circ x \frac{d}{dx}=-\left(x\dfrac{d}{dx}+\left(2k+\frac{2}{n}-1\right)\right) \circ \mathcal{F}_{k,n},\\
		\displaystyle \mathcal{F}_{k,n}\circ \mid x \mid^\frac{2}{n}I=-\mid x \mid^{2-\frac{2}{n}}\Delta_k \circ \mathcal{F}_{k,n},
		\\
		\displaystyle \mathcal{F}_{k,n} \circ \mid x \mid^{2-\frac{2}{n}}\Delta_k= -\mid x \mid^\frac{2}{n}I~\circ \mathcal{F}_{k,n}.
	\end{array}
\end{eqnarray}

Our first main theorem, proved in Subsection \ref{ProofTheorem1}, extends the invariance property $\mathcal{F}(\mathcal{S}(\R))=\mathcal{S}(\R)$ of the  classical Fourier transform $\mathcal{F}$. The proof combines the intertwining relations \eqref{knFourierIntertwining} with the algebraic characterization obtained in Subsection~\ref{knSchwartz-recursive}.
It states as follows:
\begin{thm}\label{Fkn-Schwartz-Theorem}
	For the $\kn$-generalized Fourier transform $\mathcal{F}_{k,n}$, one has  $$\mathcal{F}_{k,n}\left(\mathcal{S}_{k,n}(\R)\right)=\mathcal{S}_{k,n}(\R).$$
\end{thm}

Next,~we introduce a $\kn-$analogue of the space $\mathcal{D}(\R)$ of $C^\infty-$functions with compact support:

\begin{definition}\label{Dkn-definition}
	We introduce the space $\mathcal{D}_{k,n}(\R)$ of $C^\infty-$functions on $\R\setminus\{0\}$ with bounded support,  satisfying 
	\begin{eqnarray*}
		\sup_{\substack{x\in \R\setminus\{0\} }}~\left|\displaystyle\left(~\left|~x~\right|^{2-\frac{2}{n}}\Delta_{k}~\right)^\beta\left(x^\ell f^{(\ell)}(x)\right) \right|<\infty,&\mbox{for all}& \beta,\ell\in\N_0.
	\end{eqnarray*}
\end{definition}

In Subsection \ref{ProofTheorem2} we extend the classical embedding
\(
\mathcal{D}(\mathbb{R}) \hookrightarrow \mathcal{S}(\mathbb{R})
\)
to the spaces \(\mathcal{D}_{k,n}(\mathbb{R})\) and \(\mathcal{S}_{k,n}(\mathbb{R})\).
\begin{thm}\label{Dkn-thm}
	The embedding
	$\mathcal{D}_{k,n}(\R)\hookrightarrow \mathcal{S}_{k,n}(\R)$ is continuous.
\end{thm}

In Section \ref{Lp-section} we establish embedding and density results for $\mathcal{S}_{k,n}(\R)$ and $L^p(d\mu_{k,n})$, extending well‑known properties of the classical Schwartz space.

\begin{thm}\label{DensityThm} Let $1\leq p<\infty$. Then the following statements hold:
	\begin{itemize}
		\item[\bf (i)] The embedding
		$\mathcal{S}_{k,n}(\R)\hookrightarrow L^p(d\mu_{k,n})$ is continuous.
		\item[\bf (ii)] 	$\mathcal{S}_{k,n}(\R)$ is a dense subspace in $L^p(d\mu_{k,n})$.
	\end{itemize}
\end{thm}

In Subsection~\ref{ProofTheorem3} we prove Theorem~\ref{DensityThm} by combining Theorem~\ref{Dkn-thm} with the results of Subsection~\ref{key-subsection}. In particular, we derive the density of $\mathcal{S}_{k,n}(\mathbb{R})$ in $L^p(d\mu_{k,n})$ from the chain of continuous embeddings
\[
\mathcal{D}_{k,n}(\mathbb{R}) \hookrightarrow \mathcal{S}_{k,n}(\mathbb{R}) \hookrightarrow L^p(d\mu_{k,n}),
\]
and we also show that $\mathcal{D}_{k,n}(\mathbb{R})$ is dense in $L^p(d\mu_{k,n})$, using an approximation-of-identity theorem due to Ben Sa\"id and the second author~\cite{BeNe2}. Thus Theorems~\ref{Dkn-thm} -- \ref{DensityThm}  confirm that the $\kn-$generalized Schwartz space $\mathcal{S}_{k,n}(\R)$ enjoys invariance under $\mathcal{F}_{k,n}$ as well as embedding and density properties analogous to those of $\mathcal{S}(\R)$. 

\section{Preliminaries}\label{Preliminaries}

This section reviews key properties of the $\kn$-generalized kernel $B_{k,n}$ and the associated translation and convolution operators~\cite[Theorem 4.2]{BNS}, which are essential for proving the invariance $\mathcal{F}_{k,n}\left(\mathcal{S}_{k,n}(\R)\right)= \mathcal{S}_{k,n}(\R)$ (Theorem~\ref{Fkn-Schwartz-Theorem}) and the density of $\mathcal{S}_{k,n}(\R)$ in $L^p(d_{k,n})$ (Theorem~\ref{DensityThm}).

\subsection{The $\kn-$generalized kernel $B_{k,n}$}\label{Bkn-section}

The Schwartz kernel theorem (see \cite[Theorem 5.2.1]{Hormander03}) guarantees that the $(k,a)-$generalized Fourier transform, defined in~\eqref{FkaRN}, has an integral representation via a continuous symmetric kernel $B_{k,a}$ that satisfies
\[
B_{k,a}(0,y) = 1, \quad \text{for all } y \in \R^N.
\]

By \cite[Proposition 2.1]{GoIvTi23}, when $N=1$, $\displaystyle a=\frac{2}{n}$ ($n \in \N$), and $\displaystyle k\geq \frac{1}{2} - \frac{a}{4}$, the $\kn-$generalized Dunkl kernel $\displaystyle B_{k,n}:=B_{k,\frac{2}{n}}$ is uniformly bounded:
\begin{equation}
	\label{Bcondition}
	\forall x,y\in\mathbb{R}\qquad \left|B_{k,n}(x,y)\right|\le M,
\end{equation}
for some constant $0 < M < \infty$.


The bound~\eqref{Bcondition} yields the integral form of the $\kn-$generalized Fourier transform for $f\in L^1(d\mu_{k,n})$:
\begin{equation}
	\label{FkaRN2}
	(\mathcal{F}_{k,n}f)(y)=\int_{\mathbb{R}} f(x)\, B_{k,n}(x,y)\, d\mu_{k,n}(x), \quad y\in\mathbb{R},
\end{equation}
with the $L^1-L^\infty$ estimate
\begin{equation}
	\label{LinftyL1}
	\| \mathcal{F}_{k,n}f\|_{L^\infty(d\mu_{k,n})}\le M\, \|f\|_{L^1(d\mu_{k,n})}.
\end{equation}

Moreover, \cite[Theorem 5.3]{BKO} gives the inversion formula
\begin{equation}
	\label{FourierInversion}
	(\mathcal{F}_{k,n}^{-1}f)(x)= (\mathcal{F}_{k,n}f)\left(( -1)^n x\right),\quad x\in\mathbb{R}.
\end{equation}

Thus,
\begin{equation}
	\label{FourierInversionIntegral}
	f(x) = \int_{\mathbb{R}} (\mathcal{F}_{k,n}f)(y)\, B_{k,n}\left((-1)^n x, y\right)\, d\mu_{k,n}(y), \quad x\in\mathbb{R}.
\end{equation}

The kernel satisfies the joint eigenfunction equation~\cite[Theorem 5.7]{BKO}:
\begin{equation}
	\label{JointEigenfunction}
	|x|^{2-\frac{2}{n}} \, (\Delta_{k}B_{k,n})(x,y) = -|y|^{\frac{2}{n}} \, B_{k,n}(x,y), \quad x,y \in \mathbb{R}.
\end{equation}

As outlined in \cite[Section 4]{BKO}, it decomposes, $ x,y\in\R$, as
\begin{eqnarray}
	\label{Bkn}
	B_{k,n}(x,y) =  \j_{kn-\frac{n}{2}}\Bigl(n|xy|^{\frac{1}{n}}\Bigr) +  (-i)^n \left(\frac{n}{2}\right)^n \frac{\Gamma\Bigl(kn-\frac{n}{2}+1\Bigr)}{\Gamma\Bigl(kn+\frac{n}{2}+1\Bigr)} \, xy \, \j_{kn+\frac{n}{2}}\Bigl(n|xy|^{\frac{1}{n}}\Bigr),
\end{eqnarray}
where $\j_{\nu}$ is the normalized Bessel function: 
\begin{eqnarray}\label{normalizedBessel}
\j_{\nu}(z) = \Gamma(\nu+1)\left(\frac{z}{2}\right)^{-\nu} J_{\nu}(z),
\end{eqnarray}
with $J_\nu$ the standard Bessel function of the first kind~\cite[Chapter 1]{Tr2001Book}.

Integral representations~\cite[Chapter VII, 7.3.2 (3)]{Bateman53}, \cite[Chapter X, 10.9 (38)]{Bateman53} (or~\cite[Corollary 2.7, Remark 2.8]{GoIvTi23}) yield
\begin{eqnarray*}
	J_\nu(z)=\frac{\left(\frac{z}{2}\right)^\nu}{\sqrt{\pi}\Gamma\left(\nu+\frac{1}{2}\right)}\int_{-1}^{1}(1-t^2)^{\nu-\frac{1}{2}}e^{itz}dt,~~\Re(\nu)>-\frac{1}{2} \\ \ \\
	\int_{-1}^{1} e^{i x t}\, C_n^{\lambda}(t)\, (1-t^2)^{\lambda-\frac{1}{2}}\, dt=\\
	= i^{\,n}\,\sqrt{\pi}\,\Gamma\!\left(\lambda+\frac{1}{2}\right) \frac{\Gamma(2\lambda+n)}{n!\Gamma(2\lambda)}
	\left(\frac{2}{z}\right)^{\lambda} J_{n+\lambda}(z),~~\Re(\lambda) > -\frac{1}{2},
\end{eqnarray*}
where $C_n^\lambda$ denotes the Gegenbauer polynomials (cf. \cite[Chapter VII, 7.3.2 (3)]{Bateman53} and \cite[Chapter X, 10.9 (38)]{Bateman53}). That allows us to rewrite \eqref{Bkn} as follows (see also \cite[Corollary 2.7]{GoIvTi23} and \cite[Remark 2.8]{GoIvTi23}):
\begin{eqnarray}
	\label{BknIntegral}
B_{k,n}(x,y)=\frac{\Gamma\left(kn-\frac{n}{2}+1\right)}{\sqrt{\pi}~\Gamma\left(kn-\frac{n}{2}+\frac{1}{2}\right)}\int_{-1}^{1} {\bm b}_{k,n}(xy,t)~(1-t^2)^{kn-\frac{n}{2}-\frac{1}{2}}e^{itn|xy|^{\frac{1}{n}}}dt,
\end{eqnarray}
with $${\bm b}_{k,n}(z,t)=1+(-1)^n\frac{n!\Gamma(2\lambda)}{\Gamma(2\lambda+n)}\operatorname{sgn}(z)C_n^{kn-\frac{n}{2}}(t).$$ 


The next lemma provides polynomial growth estimates for the radial derivative of the kernel~ $B_{k,n}(x,y)$. These will later be used to justify differentiation under the integral sign and to control seminorms built from the operator $\displaystyle x\,\frac{d}{dx}.$

\begin{lemma}\label{Bkn-xdxLemma}
	For every $\ell\in \N_0$, there exists a polynomial of degree $\ell$, ${\bm M}_\ell$, with positive coefficients such that
\begin{equation}
	\label{Bcondition-ell}
\left|\left( nx\frac{d}{dx}\right)^\ell B_{k,n}(x, y)\right|\le {\bm M}_\ell\left(n|xy|^{\frac{1}{n}}\right),\qquad  	x,y\in\mathbb{R}.
\end{equation}
\end{lemma}

\begin{proof}
Let $\displaystyle u=itn|xy|^{\frac{1}{n}}$. Then $\displaystyle nx\frac{d}{dx}\left(u^j\right)=ju^j$ ($j\in\mathbb{N}_0$), since $\displaystyle x\frac{d}{dx}=|x|\frac{d}{d|x|}$. 
	
The product rule then gives
	\begin{eqnarray*}
		nx\dfrac{d}{dx}\left(u^je^{u}\right)=(ju^j+u^{j+1})~e^{u}, & j\in \mathbb{N}_0
	\end{eqnarray*}
	
Inductively, 
	$$
	\left(n x\frac{d}{dx}\right)^\ell\left(e^{u}\right)={\bm P}_\ell(u)e^{u},
	$$
holds for every $\ell\in \N_0$,	where ${\bm P}_\ell$ is a polynomial of degree $\ell\in \mathbb{N}_0$ with positive coefficients.
	
Since ${\bf b}_{k,n}(z,t)(1-t^2)^{kn - \frac{n}{2} -\frac{1}{2}} e^{i t n |z|^{\frac{1}{n}}}$ and is bounded for every $(t,z)\in [-1,1] \times K$ for compact $K$, the dominated convergence theorem justifies pulling $\left(n x \frac{d}{dx}\right)^\ell$ under the integral in \eqref{BknIntegral}. Thus,
	\begin{eqnarray}
		\label{BknIntegral-ell}
		\left(n x\frac{d}{dx}\right)^\ell	B_{k,n}(x,y) = \nonumber \\
		=\frac{\Gamma\left(kn-\frac{n}{2}+1\right)}{\sqrt{\pi}~\Gamma\left(kn-\frac{n}{2}+\frac{1}{2}\right)}\int_{-1}^{1} {\bm b}_{k,n}(xy,t)~{\bm P}_\ell\left(itn|xy|^{\frac{1}{n}}\right)(1-t^2)^{kn-\frac{n}{2}-\frac{1}{2}}e^{itn|xy|^{\frac{1}{n}}}dt.
	\end{eqnarray}

Define
	$$
	{\bm M}_\ell\left(z\right)=\left(1+\frac{n!\Gamma(2\lambda)}{\Gamma(2\lambda+n)}~|\lambda| {\bf B}(\lambda) n^{2\lambda-1}\right){\bm P}_\ell\left(z\right),~~z\in \mathbb{R},
	$$
	for some constant ${\bf B}(\lambda)>0$ depending on $\displaystyle \lambda=kn-\frac{n}{2}$.
	
	Standard estimates~(cf.~\cite[Lemma 4.9]{BKO}) provide
	\begin{eqnarray*}
\sup_{t\in [-1,1]}	\left|{\bm P}_\ell\left(itn|xy|^{\frac{1}{n}}\right)\right| \leq {\bm P}_\ell\left(n|xy|^{\frac{1}{n}}\right)
	\end{eqnarray*}
	and
			\begin{eqnarray*}
		\sup_{t\in [-1,1]}\left|\frac{1}{\lambda}C_n^\lambda(t)\right|\leq {\bf B}(\lambda)n^{2\lambda -1}, &\forall n\in \N.
	\end{eqnarray*}
	
These bounds yield \eqref{Bcondition-ell}, completing the proof of Lemma~\ref{Bkn-xdxLemma}.
\end{proof}


Building on Lemma~\ref{Bkn-xdxLemma}, the next proposition provides mixed estimates for iterates of the operators $\left|~x~\right|^{2-\frac{2}{n}}\Delta_k$ and $\displaystyle x\,\frac{d}{dx}$ applied to $B_{k,n}(x,y)$, essential for controlling derivatives of convolutions and for the seminorm estimates in Section~\ref{knSchwartz-section} and~Section~\ref{Lp-section}.

\begin{proposition}\label{Bkn-xdxProposition}
For every $\alpha,\ell\in \N_0$, there exists a polynomial of degree $\ell$, ${\bm N}_\ell$, with positive coefficients such that
	\begin{eqnarray}
			\label{Bcondition-ellDeltak}
		\left| \left(n\,|x|^{2-\frac{2}{n}}\Delta_k\right)^\alpha \left(nx\frac{d}{dx}\right)^\ell B_{k,n}(x,y)\right| \leq n^\alpha |y|^{\frac{2\alpha}{n}} {\bm N}_\ell\left(n|xy|^\frac{1}{n}\right)&,~x,y\in \R. 
	\end{eqnarray}
\end{proposition}

\begin{proof}
	The $\mathfrak{sl}(2,\mathbb{R})$ relation \(\left[\mathbb{H}_{k,\frac{2}{n}}, \mathbb{E}_{k,\frac{2}{n}}^-\right] = -2 \mathbb{E}_{k,\frac{2}{n}}^-\) (see \eqref{sl2}) implies
	$$
	\left[n x \frac{d}{dx}, \, n |x|^{2 - \frac{2}{n}} \Delta_k\right] = -2 n |x|^{2 - \frac{2}{n}} \Delta_k.
	$$
	
The graded Leibniz rule for commutators then yields
	$$
	\left[n x \frac{d}{dx}, \left(n |x|^{2 - \frac{2}{n}} \Delta_k\right)^\alpha\right] = -2\alpha \left(n |x|^{2 - \frac{2}{n}} \Delta_k\right)^\alpha.
	$$
	
	Rearranging yields the intertwining form
	$$
	\left(n |x|^{2 - \frac{2}{n}} \Delta_k\right)^\alpha \biggl( n x \frac{d}{dx} \biggr)
	= \biggl( n x \frac{d}{dx} + 2\alpha I \biggr) \left(n |x|^{2 - \frac{2}{n}} \Delta_k\right)^\alpha.
	$$
	
	By induction on $\ell \in \mathbb{N}_0$, this extends to
	$$
	\left(n |x|^{2 - \frac{2}{n}} \Delta_k\right)^\alpha \biggl( n x \frac{d}{dx} \biggr)^\ell
	= \biggl( n x \frac{d}{dx} + 2\alpha I \biggr)^\ell \left(n |x|^{2 - \frac{2}{n}} \Delta_k\right)^\alpha.
	$$
	
	From \eqref{JointEigenfunction} it follows that
	$$
	\left(n |x|^{2 - \frac{2}{n}} \Delta_k\right)^\alpha B_{k,n}(x,y) = (-n)^\alpha |y|^{\frac{2\alpha}{n}} B_{k,n}(x,y).
	$$
	
	Thus,
	\begin{align*}
		\biggl( n |x|^{2 - \frac{2}{n}} \Delta_k \biggr)^\alpha \biggl( n x \frac{d}{dx} \biggr)^\ell B_{k,n}(x,y)
		&= (-n)^\alpha |y|^{\frac{2\alpha}{n}} \biggl( n x \frac{d}{dx} + 2\alpha I \biggr)^\ell B_{k,n}(x,y) \\
		&= (-n)^\alpha |y|^{\frac{2\alpha}{n}} \sum_{j=0}^\ell \binom{\ell}{j} (2\alpha)^{\ell-j} \biggl( n x \frac{d}{dx} \biggr)^j B_{k,n}(x,y).
	\end{align*}
	
	By the triangle inequality and Lemma~\ref{Bkn-xdxLemma},
	$$
	\biggl| \biggl( n |x|^{2 - \frac{2}{n}} \Delta_k \biggr)^\alpha \biggl( n x \frac{d}{dx} \biggr)^\ell B_{k,n}(x,y) \biggr|
	\leq n^\alpha |y|^{\frac{2\alpha}{n}} \sum_{j=0}^\ell \binom{\ell}{j} (2\alpha)^{\ell-j} {\bm M}_j \bigl( n |xy|^{\frac{1}{n}} \bigr).
	$$
	
	Thus, the inequality~\eqref{Bcondition-ellDeltak} holds with the polynomial
$$
{\bm N}_\ell(z) = \sum_{j=0}^\ell \binom{\ell}{j} (2\alpha)^{\ell-j}
{\bm M}_j\!\left(z\right).
$$
\end{proof}

\subsection{Translation and convolution operators for $\mathcal{F}_{k,n}$}\label{Convolution-section}
According to \cite[Theorem 4.2]{BNS} and under the condition
\(
kn - \frac{n}{2} > -\frac{1}{2},
\)~
the product formula for two \(\kn\)-generalized Dunkl kernels can be written as
	\begin{equation}\label{kernel}
		B_{k,n}(x,w)B_{k,n}(y,w) = \int_\mathbb{R} B_{k,n}(z,w) \,d\nu_{x,y}^{k,n}(z),
	\end{equation}
	where the measure \(d\nu_{x,y}^{k,n}(z)\) is defined by
	\begin{equation}\label{Kkernel} d\nu_{x,y}^{k,n}(z)=\begin{cases}
			\mathcal{K}_{k,n}(x,y,z)d\mu_{k,n}(z) &,~\textrm{if}~xy\neq 0 \\
			d\delta_x(z) &,~\textrm{if}~y=0 \\
			d\delta_y(z) &,~\textrm{if}~x=0.
		\end{cases}
	\end{equation} 
	
	Here, \(\mathcal{K}_{k,n}(x,y,z)\) denotes a non-positive kernel whose support is contained in the set
	$$I_{x,y}=\left\{z\in \mathbb{R}~:~\; \left||x|^\frac{1}{n}-|y|^\frac{1}{n}\right|<|z|^\frac{1}{n}<|x|^\frac{1}{n}+|y|^\frac{1}{n}\right\}.$$
	
In a natural way, the generalized translation operator $\tau^{k,n}_{x}$ is then introduced by its integral representation for suitable functions:
	\begin{equation}\label{translation}(\tau^{k,n}_{x}f)(y)=\int_{\mathbb{R}}f(z)\,d\nu_{x,y}^{k,n}(z),\quad x\in\mathbb{R}.
	\end{equation}
	
For each $\:x,\:y,\:z\in \mathbb{R}$, we immediately have
	\begin{eqnarray*}
		(\tau^{k,n}_{x}B_{k,n})(y,z)=B_{k,n}(x,z)\,B_{k,n}(y,z), 
	\end{eqnarray*}
	and
	\begin{equation*}
		\mathcal{F}_{k,n}(\tau^{k,n}_{x}f)(y)= B_{k,n}((-1)^{n} x,y)\mathcal{F}_{k,n}f(y).\end{equation*}
		
		The definition of the generalized translation operator \(\tau^{k,n}_{x}\), as presented in \eqref{translation}, extends naturally to the space \(L^p(d\mu_{k,n})\). In particular, for every \(x \in \mathbb{R}\) and \(1 \leqslant p \leqslant \infty\), the following bound holds:
			\begin{equation}\label{normtrans}
			\|\tau^{k,n}_{x} f\|_{L^{p}(d\mu_{k,n})}\leqslant A_{k,n}\| f\|_{L^{p}(d\mu_{k,n})}, 
		\end{equation}
		where the constant \(A_{k,n}\) satisfies \(0 < A_{k,n} < \infty\) and depends solely on the parameters \(k\) and \(n\).
		
The \(\kn\)-generalized convolution product of two appropriate functions \(f\) and \(g\) is defined via the generalized translation operator as
\begin{eqnarray}\label{convolutionIntegral}
	(f\star_{k,n} g)(x)=\int_{\mathbb{R}} f(y) \, \left(\tau_{x}^{k,n} g\right)\Big((–1)^n y\Big)\, d\mu_{k,n}(y), \quad x\in \mathbb{R}.
\end{eqnarray}

{This convolution operation seamlessly integrates with the \(\kn-\)generalized Fourier transform defined in \eqref{FkaRN2}. It}
may be rewritten as
\begin{equation}\label{ConvolutionFormula}
	\left(f\star_{k,n} g\right)(x)=\int_{\mathbb{R}} \big(\mathcal{F}_{k,n} f\big)(y) \, \big(\mathcal{F}_{k,n} g\big)(y) \, B_{k,n}\Big((–1)^n x,y\Big) \, d\mu_{k,n}(y), \quad x\in \mathbb{R}.
\end{equation}

The convolution structure satisfies Young's inequality. 
 Specifically, for functions \(f\in L^{p}(d\mu_{k,n})\) and \(g\in L^{r}(d\mu_{k,n})\) with \(1\leq p,q,r\leq \infty\) and satisfying
\(\displaystyle
\frac{1}{p}+\frac{1}{r}=\frac{1}{q}+1,
\)
the following estimate holds:
\begin{eqnarray}\label{YoungIneq}
	\|f\star_{k,n} g\|_{L^{q}(d\mu_{k,n})} \leq A_{k,n}\,\|f\|_{L^{p}(d\mu_{k,n})} \, \|g\|_{L^{r}(d\mu_{k,n})},
\end{eqnarray}
where \(A_{k,n}\) is a positive constant depending only on \(k\) and \(n\) (see \cite[Properties 5.6]{BNS}).

In particular, if \(1\leq p,q,r\leq 2\) are such that
\(\displaystyle 
\frac{1}{p}+\frac{1}{r}=\frac{1}{q}+1,
\)
and if \(f\in L^p(\mu_{k,n})\) and \(g\in L^r(\mu_{k,n})\), then the generalized Fourier transform obeys

\begin{equation}\label{FourierknConvolution}
	\mathcal{F}_{k,n}\big(f\star_{k,n} g\big) = \big(\mathcal{F}_{k,n} f\big) \cdot \big(\mathcal{F}_{k,n} g\big).
\end{equation}

\section{The spaces $\mathcal{S}_{k,n}(\R)$ and $\mathcal{D}_{k,n}(\R)$}
\label{knSchwartz-section}

Before proving Theorem \ref{Fkn-Schwartz-Theorem} and Theorem \ref{Dkn-thm}, we first refine the description of $\mathcal{S}_{k,n}(\mathbb{R})$ from Definition~\ref{knSchwartz}, by rewriting the seminorm condition \eqref{sup-kn} in terms of the generators of $\mathfrak{sl}(2,\mathbb{R})$ given by \eqref{LadderOp}.

  The technical lemmas in Subsection \ref{knSchwartz-recursive}  normal-ordering formulas needed to control these seminorms and are used systematically in Subsections \ref{ProofTheorem1} and \ref{ProofTheorem2}.
  
\subsection{Technical results}\label{knSchwartz-recursive}

The next lemma states a key recursive formula involving the operators $\displaystyle n\left|~x~\right|^{2-\frac{2}{n}}\Delta_k$ and $n\left|~x~\right|^{\frac{2}{n}}$, which underlie the normal-ordering formulas used later to rewrite the defining seminorms of $\mathcal{S}_{k,n}(\mathbb{R})$ in terms of the generator $\displaystyle \mathbb{H}_{k,\frac{2}{n}}$ of $\mathfrak{sl}(2,\R)$ (see \eqref{LadderOp}).
\begin{lemma}\label{RecursiveLemma} For every $f\in C^\infty(\R\setminus \{0\})$ and $\beta \in \N_0$, one has
	\begin{align}\label{Recursive1}
		\begin{array}{lll}
	n\left|~x~\right|^{2-\frac{2}{n}}\Delta_k\left( \left(n\left|~x~\right|^{\frac{2}{n}}\right)^\beta f(x)\right)=\\ =4\beta\left(n\left|~x~\right|^{\frac{2}{n}}\right)^{\beta-1}\left((\beta-1)f(x)+\mathbb{H}_{k,\frac{2}{n}}f(x)\right)  \\ +\left(n\left|~x~\right|^{\frac{2}{n}}\right)^\beta\left(n\left|~x~\right|^{2-\frac{2}{n}}\Delta_k f(x)\right),~x\in \R \setminus \{0\}.
		\end{array}
	\end{align}
\end{lemma}

\begin{proof}
Recall that $\mathbb{E}_{k,\frac{2}{n}}^+,\mathbb{E}_{k,\frac{2}{n}}^-$ and $\mathbb{H}_{k,\frac{2}{n}}$, defined in \eqref{LadderOp}, satisfy the commutation relations \eqref{sl2}.

\underline{Step 1 -- action of $\mathbb{H}_{k,\frac{2}{n}}$ on $\left(n\left|~x~\right|^{\frac{2}{n}}\right)^{\beta}$:} 

Since \(\left[\mathbb{H}_{k,\frac{2}{n}},\mathbb{E}_{k,\frac{2}{n}}^+\right]=2\mathbb{E}_{k,\frac{2}{n}}^+\) is equivalent to
\(\mathbb{H}_{k,\frac{2}{n}}\mathbb{E}_{k,\frac{2}{n}}^+ = \mathbb{E}_{k,\frac{2}{n}}^+\left(2I+\mathbb{H}_{k,\frac{2}{n}}\right)\), we obtain
	\begin{eqnarray*}
		\mathbb{H}_{k,\frac{2}{n}}\left(n\left|~x~\right|^{\frac{2}{n}}f(x)\right)=n\left|~x~\right|^{\frac{2}{n}}\left(2f(x)+\mathbb{H}_{k,\frac{2}{n}}f(x)\right), & x\in \R \setminus \{0\}.
	\end{eqnarray*}
	
An induction on $j\in \N_0$ then yields
	\begin{eqnarray}
		\label{intertwiningH_beta}
		\mathbb{H}_{k,\frac{2}{n}}\left(\left(n\left|~x~\right|^{\frac{2}{n}}\right)^j f(x)\right)=\left(n\left|~x~\right|^{\frac{2}{n}}\right)^j\left((2j) f(x)+\mathbb{H}_{k,\frac{2}{n}}f(x)\right), &x\in \R \setminus \{0\}.
	\end{eqnarray}
	
	\underline{Step 2 -- commutator $\left[~n\left|~x~\right|^{2-\frac{2}{n}}\Delta_k,\left(n\left|~x~\right|^{\frac{2}{n}}\right)^\beta\right]$:} 
	
Next, from \(\left[\mathbb{E}_{k,\frac{2}{n}}^+,\mathbb{E}_{k,\frac{2}{n}}^-\right]=\mathbb{H}_{k,\frac{2}{n}}\) we deduce
	\begin{eqnarray}
		\label{commutatorH_E}\left[~n\left|~x~\right|^{2-\frac{2}{n}}\Delta_k,n\left|~x~\right|^{\frac{2}{n}}I~\right]=\left[-2i\mathbb{E}_{k,\frac{2}{n}}^-,-2i\mathbb{E}_{k,\frac{2}{n}}^+\right]=4\mathbb{H}_{k,\frac{2}{n}}.
	\end{eqnarray}
	

Then, for each $\beta \in \N$ the Lie bracket identity
\begin{eqnarray*}
[A,B^\beta]=\sum_{j=0}^{\beta-1} B^{\beta-1-j}[A,B]B^{j}
\end{eqnarray*}
applied to $\displaystyle A=n\left|~x~\right|^{2-\frac{2}{n}}\Delta_k$ and $\displaystyle B=n\left|~x~\right|^{\frac{2}{n}}$, combined with \eqref{commutatorH_E} and \eqref{intertwiningH_beta} gives
	\begin{align*}
		\left[~n\left|~x~\right|^{2-\frac{2}{n}}\Delta_k,\left(n\left|~x~\right|^{\frac{2}{n}}\right)^\beta I~\right]f(x)=\nonumber \\
		=4\beta\left(n\left|~x~\right|^{\frac{2}{n}}\right)^{\beta-1}\left((\beta-1)f(x)+\mathbb{H}_{k,\frac{2}{n}}f(x)\right),~~~x\in \R.
	\end{align*}
	
By the definition of the commutator,
	\begin{eqnarray*}
		\left[~n\left|~x~\right|^{2-\frac{2}{n}}\Delta_k,\left(n\left|~x~\right|^{\frac{2}{n}}\right)^\beta I~\right]f(x)= \\
		=~n\left|~x~\right|^{2-\frac{2}{n}}\Delta_k\left( \left(n\left|~x~\right|^{\frac{2}{n}}\right)^\beta f(x)\right)-\left(n\left|~x~\right|^{\frac{2}{n}}\right)^\beta\left(n\left|~x~\right|^{2-\frac{2}{n}}\Delta_k f(x)\right),~~~x\in \R\setminus \{0\},
	\end{eqnarray*} which is exactly \eqref{Recursive1}, completing the proof of Lemma \ref{RecursiveLemma}.
\end{proof}

The next lemma introduces a family of seminorms adapted to $\mathcal{S}_{k,n}(\mathbb{R})$ and connects them to powers of $\mathbb{H}_{k,\frac{2}{n}}$, using Stirling numbers of the second kind and Katriel’s boson normal ordering approach~\cite{Katriel74}. This result provides an $\mathfrak{sl}(2,\R)-$based characterization of the topology on $\mathcal{S}_{k,n}(\mathbb{R})$, which will be crucial in the proof of Theorem~\ref{Fkn-Schwartz-Theorem}.

\begin{lemma}\label{Schwartz-Lemma}
Let $\left(P_{\alpha,\beta}\right)_{\alpha,\beta\in \N_0}$ be the family of seminorms
\begin{eqnarray}
	\label{seminormP}
	P_{\alpha,\beta}(g)=\sup_{x\in \mathbb{R}\setminus\{0\}} \Bigl| \Bigl( n|~x~|^{\frac{2}{n}} \Bigr)^\alpha \Bigl( n|~x~|^{2-\frac{2}{n}} \Delta_k \Bigr)^\beta g(x) \Bigr|,
\end{eqnarray}
and let $(f_m)_{m\in \N_0}$ be the sequence of functions given by \begin{eqnarray}
		\label{fm}f_m(x)
		=\sum_{\ell=0}^m \sum_{j=0}^{\ell} \binom{m}{\ell}\left(kn+1-\frac{n}{2}\right)^{m-\ell} n^\ell S(\ell,j)~ x^jf^{(j)}(x), & x\in \R,
	\end{eqnarray}
	where $S(\ell,j)$ denotes the Stirling numbers of the second kind (cf~\cite[\S~ 24.1.4]{AS86}).
	
The following assertions are then true:
	\begin{enumerate}
		\item[{\bf (a)}] $f\in \mathcal{S}_{k,n}(\mathbb{R})$ if and only if the seminorm condition
			\begin{eqnarray}
			\label{sup-fm}
			P_{\alpha,\beta}(f_m)<\infty, &\mbox{for all}& \alpha,\beta,m \in \N_0.	
			\end{eqnarray}
			is always satisfied.
		\item[{\bf (b)}] For every $m \in \mathbb{N}_0$, we have the mapping property
		\begin{eqnarray}
			\label{fm-Mapping}	\left(\mathbb{H}_{k,\frac{2}{n}}\right)^m:f\mapsto f_{m}.
		\end{eqnarray}
	\end{enumerate}
\end{lemma}

\begin{proof}
By Definition~\ref{knSchwartz} and \eqref{seminormP} it easily follows that a function \(f\) belongs to \(\mathcal{S}_{k,n}(\mathbb{R})\) if and only if
	\begin{eqnarray}
	\label{supXjDj}	P_{\alpha,\beta}(X^j D^j f) < \infty, & \text{for all}& \alpha,\beta,j \in \N_0,
	\end{eqnarray}
where 
\begin{eqnarray}
	\label{X-D}
	X = xI \quad\text{and}\quad D = \frac{d}{dx}.
\end{eqnarray}

\underline{Proof of \textbf{(a)}:}

From \eqref{fm} the inequality
\begin{eqnarray}
	\label{IneqP1}
P_{\alpha,\beta}(f_m) \leq \sum_{\ell=0}^{m} \sum_{j=0}^{\ell} \binom{m}{\ell}  \Bigl(kn + 1 - \frac{n}{2}\Bigr)^{m-\ell} n^{\ell}S(\ell,j)\, P_{\alpha,\beta}(X^j D^j f) 
\end{eqnarray} 
follows immediately from the triangle inequality.

Conversely, $n^m X^m D^m f(x)$ equals to the $(m,m)-$coefficient of the summand in~\eqref{fm}, since 
\[
S(m,m)=1 \quad \text{and} \quad \binom{m}{m}=1.
\]

Furthermore, since all the coefficients in \eqref{fm} are positive, the inequality 
\begin{eqnarray}
	\label{IneqP2}P_{\alpha,\beta}(X^m D^m f) \leq n^{-m} P_{\alpha,\beta}(f_m)
\end{eqnarray}
follows straightforwardly.

Thus, \eqref{IneqP1} and \eqref{IneqP2} show that $P_{\alpha,\beta}(X^m D^m f)$ and $P_{\alpha,\beta}(f_m)$ define equivalent seminorms. This proves condition \eqref{sup-fm}.

\underline{Proof of {\bf (b)}:} 

We begin by observing that the ladder operators $X$ and $D$ defined by \eqref{X-D} are canonical generators of the boson algebra -- i.e, the Weyl-Heisenberg algebra of dimension $3$ -- so that the {\it boson normal ordering approach} can be used to formally represent the iterated powers $\displaystyle \left(x\frac{d}{dx}\right)^\ell$, $\ell \in \N$, of $\displaystyle x\frac{d}{dx}=XD$ as follows~(cf.~\cite{Katriel74}): 
\begin{eqnarray}
	\label{NormalOrdering}
	\left(x\frac{d}{dx}\right)^\ell f(x)=\sum_{j=1}^{\ell} S(\ell,j)~ X^jD^jf(x), & x\in \R.
\end{eqnarray}

So, using the fact that $S(\ell,\ell)=1$, $S(\ell,j)=0$ for $\ell<j$, and $S(\ell,0)$ for $\ell\geq 1$ (see~\cite[\S~24.1.4]{AS86}), we find that
\eqref{fm} admits the binomial representation
\begin{eqnarray}
	\label{fm-binomial}
	f_m(x)=\sum_{\ell=0}^m \binom{m}{\ell}\left(kn+1-\frac{n}{2}\right)^{m-\ell} n^\ell \left(x\frac{d}{dx}\right)^\ell f(x),& x\in \R.
\end{eqnarray}

Since the two summands of
\[
\mathbb{H}_{k,\frac{2}{n}}
= n x\frac{d}{dx}+\Bigl(kn+1-\frac{n}{2}\Bigr)I
\]
 commute, the right-hand side of \eqref{fm-binomial} is exactly the binomial expansion
 of $\left(\mathbb{H}_{k,\frac{2}{n}}\right)^{\,m}f$. This proves \eqref{fm-Mapping}, as desired.
\end{proof}

We now relate Lemma~\ref{RecursiveLemma} and Lemma~\ref{Schwartz-Lemma}. The result is the following proposition:

\begin{proposition}\label{Schwartz-Proposition}
For each $\beta\in \N_0$, let $\left(\widetilde{f}_{\beta,\ell}\right)_{\ell \leq \beta}$ be the sequence of functions defined by
	\begin{eqnarray}
		\label{fl-Stirling}
		\widetilde{f}_{\beta,\ell}(x)=(-1)^\ell \sum_{j=0}^\ell \sum_{m=0}^j s(\ell,j)\binom{j}{m}\beta^{j-m}f_m(x) &,~x\in \R,
	\end{eqnarray}
	where $(f_m)_{m\in \N_0}$ denotes the sequence of functions defined by~\eqref{fm}, and $s(\ell,j)$ denotes the signed Stirling numbers of the first kind (cf.~\cite[subsection 24.1.3]{AS86}).

Then we have the following:
	\begin{enumerate}
		\item[{\bf (A)}] $f\in \mathcal{S}_{k,n}(\R)$ if and only if 
		\begin{eqnarray}
			\label{sup-flbeta}
		P_{\alpha,\beta}\left(\widetilde{f}_{\beta,\ell}\right)<\infty, &\mbox{for all}& \alpha,\beta,\ell\in \N_0,~~\mbox{such that}~~\ell\leq \beta,	\end{eqnarray}\label{fl-Schwartzkn}
		where $P_{\alpha,\beta}$ defines the seminorm defined by \eqref{seminormP}.
		\item[{\bf (B)}] $\left(\widetilde{f}_{\beta, \ell}\right)_{\ell\leq \beta}$ satisfies
		\begin{eqnarray}
			\label{fl-Recursive}	\widetilde{f}_{\beta, \ell}=\begin{cases}
				f&,~\ell=0 \\
				\left((\beta-\ell) I+\mathbb{H}_{k,\frac{2}{n}}\right)\widetilde{f}_{\beta,\ell-1}&,~\ell\geq 1.
			\end{cases} 
		\end{eqnarray}
	\end{enumerate}
\end{proposition}

\begin{proof}
First, we note that the proof of the statement {\bf (A)} follows from the statement {\bf (a)} of the Lemma \ref{Schwartz-Lemma} and from the set of inequalities
\begin{eqnarray*}
	\label{fl-Ineq}
	P_{\alpha,\beta}\left({f_{\ell}}\right)\leq 	P_{\alpha,\beta}\left(\widetilde{f}_{\beta,\ell}\right)\leq \sum_{j=0}^\ell \sum_{m=0}^j \left|s(\ell,j)\right|\binom{j}{m}\beta^{j-m}~	P_{\alpha,\beta}\left(f_{m}\right),
\end{eqnarray*}
where $P_{\alpha,\beta}$ denotes the seminorm \eqref{seminormP} considered in Lemma \ref{Schwartz-Lemma}.
	
	For the proof of statement {\bf (B)}, let us denote by $(\lambda)_\ell$ ($\ell\in \N_0$) the falling factorials
	\begin{eqnarray}
\label{FallingFactorials}
	(\lambda)_{\ell}=
	\begin{cases}
		1 &,~\ell=0 \\
		\displaystyle \prod_{j=0}^{\ell-1}(\lambda-j)&,~\ell\geq 1. 
	\end{cases}
	\end{eqnarray}
	
	Using induction arguments over $\ell\in \N_0$, we prove that every $\widetilde{f}_{\beta,\ell}$ admits the closed formula
	{\begin{eqnarray}
		\label{fl-FallingFactorials}
		{\widetilde{f}_{\beta,\ell}(x)=	(-1)^\ell\left(-\beta I-\mathbb{H}_{k,\frac{2}{n}}\right)_{\ell}f(x), }& x\in \R,
	\end{eqnarray}}
		where $\left(-\beta I-\mathbb{H}_{k,\frac{2}{n}}\right)_{\ell}$ denotes the operator obtained by formally substituting  $\lambda\mapsto -\beta I-\mathbb{H}_{k,\frac{2}{n}}$ on \eqref{FallingFactorials}.

Then, the equivalence between \eqref{fl-Stirling} and~\eqref{fl-Recursive} follows from the identity
	$$ (\lambda)_\ell=\sum_{j=0}^\ell s(\ell,j)\lambda^j,$$
where $s(\ell,j)$ denotes the signed Stirling numbers of the first kind (cf.~\cite[\S~24.1.3]{AS86}).
\end{proof}

We now have the key ingredients needed to prove Theorem \ref{Fkn-Schwartz-Theorem} in the next subsection.

\subsection{Proof of Theorem \ref{Fkn-Schwartz-Theorem}}\label{ProofTheorem1}

Combining the intertwining properties \eqref{knFourierIntertwining} of the $\kn-$generalized Fourier transform with Lemma \ref{RecursiveLemma}, Lemma \ref{Schwartz-Lemma} and Proposition~\ref{Schwartz-Proposition} yields the following proof for Theorem \ref{Fkn-Schwartz-Theorem}:
\begin{proof}
	We begin by recalling the intertwining properties 	\eqref{knFourierIntertwining}  (see also~\cite[Theorem 5.6]{BKO}).
	Consequently, it follows immediately that
	\begin{eqnarray*}
		\left(n\left|~y~\right|^{\frac{2}{n}}~\right)^{\alpha} \left(n\left|~y~\right|^{2-\frac{2}{n}}\Delta_{k}~\right)^\beta (\mathcal{F}_{k,n}f)(y)=(-1)^{\alpha+\beta}\int_\R  h_{\alpha,\beta}(x)~B_{k,n}(x,y) d\mu_{k ,n}(x),
	\end{eqnarray*}
	holds for all $\alpha,\beta\in \N_0$, where
	$\left(h_{\alpha,\beta}\right)_{\alpha,\beta\in \N_0}$ denotes the sequence of functions defined by
	\begin{eqnarray*}
		h_{\alpha,\beta}(x)&=&\left(~n\left|~x~\right|^{2-\frac{2}{n}}\Delta_{k}~\right)^\alpha\left(~n\left|~x~\right|^{\frac{2}{n}}~\right)^{\beta} f(x),~~~x\in\R\setminus \{0\}.
	\end{eqnarray*}
	
	Thus, by the {\it uniform boundedness} of the kernel $B_{k,n}(x,y)$, provided by \eqref{Bcondition}, together with the inversion formula \eqref{FourierInversionIntegral} and the intertwining property $\mathcal{F}_{k,n}\circ \mathbb{H}_{k,\frac{2}{n}}=-\mathbb{H}_{k,\frac{2}{n}}\circ \mathcal{F}_{k,n}$ (which also follows from \eqref{knFourierIntertwining}), it suffices to verify that 
	\begin{eqnarray}
		\label{L1-integral}
		\int_\R \left|~h_{\alpha,\beta}(x)~\right| d\mu_{k ,n}(x)<\infty,
	\end{eqnarray}
	holds for each $f\in \mathcal{S}_{k,n}(\R)$.

	We first observe that the recursive relation
	\begin{eqnarray*}
		h_{\alpha,\beta}(x)=\left(~n\left|~x~\right|^{2-\frac{2}{n}}\Delta_{k}~\right)^{\alpha-\beta}h_{\beta,\beta}(x)&,~x\in \R\setminus \{0\},
	\end{eqnarray*}
	which holds for $\alpha>\beta$, allows us to restrict ourselves to the case $\alpha \leq \beta$.
	
For $\alpha=0$, the proof is immediate by the definition. Otherwise, by combining Lemma \ref{RecursiveLemma} and Lemma~\ref{Schwartz-Lemma}, one may show that $h_{1,\beta}\in \mathcal{S}_{k,n}(\R)$. Indeed, since $\mathbb{H}_{k,\frac{2}{n}}f \in \mathcal{S}_{k,n}(\R)$ and
	\begin{eqnarray*}
		h_{1,\beta}(x)=4\beta\left(~n\left|~x~\right|^{\frac{2}{n}}~\right)^{\beta-1}\left((\beta-1) f(x)+\mathbb{H}_{k,\frac{2}{n}}f(x)\right)~\\+\left(~n\left|~x~\right|^{\frac{2}{n}}~\right)^{\beta}\left(~n\left|~x~\right|^{2-\frac{2}{n}}\Delta_{k}~\right)f(x), &x\in \R,
	\end{eqnarray*}
	the desired result is obtained by applying the triangle inequality.
	
	For a general $\alpha\in \N_0$,~an inductive argument based on Lemma \ref{RecursiveLemma} and Lemma~\ref{Schwartz-Lemma} shows that there exist constants $\lambda_{\beta,\ell}\in \R$ ($0\leq \ell \leq \alpha \leq \beta$) such that 
	\begin{eqnarray}
		\label{h_alpha-beta}h_{\alpha,\beta}(x)=
		\displaystyle \sum_{\ell=0}^{\alpha}\lambda_{\beta,\ell}~\left(~n\left|~x~\right|^{\frac{2}{n}}~\right)^{\beta-\alpha-\ell}\left(~n\left|~x~\right|^{2-\frac{2}{n}}\Delta_{k}~\right)^\ell \widetilde{f}_{\beta,\ell}(x)&,~x\in \R.
	\end{eqnarray}
	
	Here, the sequence $\left(\widetilde{f}_{\beta,\ell}\right)_{\ell \leq \beta}$ is defined as in Proposition \ref{Schwartz-Proposition}.
	The triangle inequality, when applied to the expansion \eqref{h_alpha-beta} directly yields $h_{\alpha,\beta}\in \mathcal{S}_{k,n}(\R)$. 
	Finally, in order to verify the condition \eqref{L1-integral}, we choose $\nu\in \N$ such that the integral 
	$$
	\sigma_{k,n}(\nu):=\int_{\R} \left(1+n\left|~x~\right|^{\frac{2}{n}}\right)^{-\nu}d\mu_{k ,n}(x)
	$$
	converges. Since $h_{\alpha,\beta}\in \mathcal{S}_{k,n}(\R)$, one has
	\begin{eqnarray*}
		\label{supCond}
		\sup_{x\in \R \setminus\{0\}}\left(1+n\left|~x~\right|^{\frac{2}{n}}\right)^{\nu}\left|h_{\alpha,\beta}(x)\right|<\infty,& \mbox{for all}&\alpha,\beta\in \N_0.
	\end{eqnarray*}
	
	Hence,
	\begin{eqnarray*}
		\label{L1-integralg}
		\int_\R \left|~h_{\alpha,\beta}(x)~\right| d\mu_{k ,n}(x) \leq \sigma_{k,n}(\nu) ~\sup_{x\in \R\setminus\{0\}}\left(1+n\left|~x~\right|^{\frac{2}{n}}\right)^{\nu}\left|h_{\alpha,\beta}(x)\right|,
	\end{eqnarray*}
	which verifies the condition \eqref{L1-integral} and completes the proof.
\end{proof}


\subsection{Proof of Theorem \ref{Dkn-thm}}\label{ProofTheorem2}


We now prove the continuity of the embedding \(
\mathcal{D}_{k,n}(\mathbb{R}) \hookrightarrow \mathcal{S}_{k,n}(\mathbb{R})
\):

\begin{proof}



 Let $(f_m)_{m\in \mathbb{N}_0}$ be the sequence of functions defined in Lemma~\ref{Schwartz-Lemma}. Since both operators $\displaystyle x\frac{d}{dx}$ and  $\left|~\cdot~\right|^{2-\frac{2}{n}}\Delta_k$ map {\it functions with bounded support} to {\it functions with bounded support}, if $f\in \mathcal{D}_{k,n}(\R)$ satisfies the support condition $\operatorname{supp}(f)\subseteq [-R,R]$ for some $R>0$, then
	\begin{eqnarray*}
		\operatorname{supp}\left(\left(x\frac{d}{dx}\right)^j f\right)\subseteq [-R,R], & \mbox{holds for every} & j\in \N_0.
	\end{eqnarray*}
	
	An inductive application of the statement \textbf{(b)} in Lemma~\ref{Schwartz-Lemma} then shows that
	\begin{eqnarray*}
		\operatorname{supp}\left(\left(n\left|~\cdot~\right|^{2-\frac{2}{n}}\Delta_k\right)^\beta f_m\right)\subseteq [-R,R], & \mbox{for every} & \beta,m\in \N_0.
	\end{eqnarray*}
	
	So, for each $\alpha,\beta,m\in \N_0$ and $x\in \R$ we have
	$$ \left|~\left(~n\left|~x~\right|^{\frac{2}{n}}~\right)^{\alpha} \left(~n\left|~x~\right|^{2-\frac{2}{n}}\Delta_{k}~\right)^\beta f_m(x)~\right|\leq n^\alpha R^{\frac{2\alpha}{n}} \left| \left(~n\left|~x~\right|^{2-\frac{2}{n}}\Delta_{k}~\right)^\beta f_m(x)~\right|.$$ 
	
Therefore, we conclude that
\begin{eqnarray*}
	P_{\alpha,\beta}(f_m)\leq n^\alpha R^{\frac{2\alpha}{n}}  P_{0,\beta}(f_m), &\mbox{for each}&\alpha,\beta,m\in \N_0,
\end{eqnarray*}
where $\left(P_{\alpha,\beta}\right)_{\alpha,\beta\in \mathbb{N}_0}$ denotes the family of seminorms defined by \eqref{seminormP} in Lemma~\ref{Schwartz-Lemma}. These seminorms determine the topology of the spaces $\mathcal{D}_{k,n}(\R)$ and $\mathcal{S}_{k,n}(\R)$, respectively (see Lemma~\ref{Schwartz-Lemma}). Consequently, the above inequality shows that the inclusion \(
\mathcal{D}_{k,n}(\mathbb{R}) \hookrightarrow \mathcal{S}_{k,n}(\mathbb{R})
\) is continuous with respect to the seminorms $(P_{\alpha,\beta})_{\alpha,\beta\in \N_0}$, which completes the proof of Theorem \ref{Dkn-thm}.

\end{proof}

\section{Density in $L^p(d\mu_{k, n})$}\label{Lp-section}

\subsection{Preliminary results}\label{key-subsection}

In order to prove Theorem \ref{DensityThm}, it is necessary to use an approximation of the identity with a convolution structure generated by a contractive translation operator. In view of \eqref{normtrans} and \eqref{convolutionIntegral}, we restrict to the case $\displaystyle a=\frac{2}{n}$.

We recall that a family $(\varphi_r)_{r>0}$ is called an approximate identity if and only if it satisfies:
\begin{itemize}
	\item[(a)] $\varphi_r\in L^1(d\mu_{k,n})$ for all $r>0$;
	\item[(b)] $\displaystyle \int_{\mathbb{R}}\varphi_r(x)\,d\mu_{k,n}(x)=1$ for every $r>0$;
	\item[(c)] For any $\delta>0$, 
	\[
	\lim_{r\to0}\int_{|x|\geq\delta}|\varphi_r(x)|\,d\mu_{k,n}(x)=0.
	\]
\end{itemize}

One way to construct an approximate identity is similar to the construction considered by the second author in collaboration with Ben Sa\"id in \cite{BeNe2}. Namely, we choose non-negative function $\varphi\in L^1(d\mu_{k,n})$ satisfying
\[
\int_{\mathbb{R}}\varphi(x)\,d\mu_{k,n}(x)=1,
\]
the family of functions $(\varphi_r)_{r>0}$ is defined, for each $r>0$, by the formula
\begin{eqnarray}
	\label{rDilation}
	\varphi_r(x)=r^{-\left(2k+\frac{2}{n}-1\right)}\varphi\left(\frac{x}{r}\right), & x\in \R.
\end{eqnarray}

For our purposes, the following result is also required for the proof of Theorem \ref{DensityThm}:
\begin{thm}[Theorem 3.1 of \cite{BeNe2}]\label{BeNe2-Thm}
	Suppose that $(\varphi_r)_{r>0}$ is an approximation identity as in \eqref{rDilation}.
	\begin{itemize}
		\item[\bf (1)] If $f\in L^p(d\mu_{k, n})$, with $1\leq p<\infty$, then $f\star_{k ,n}\varphi_{r}\in L^p(d\mu_{k, n})$ and $f\star_{k ,n}\varphi_{r}\rightarrow f$ in the $L^p-$norm as $r\rightarrow 0$.
		\item[\bf (2)] If $f\in C_0(\R)$, then $f\star_{k ,n}\varphi_{r}\in L^\infty(d\mu_{k, n})$ and $f\star_{k ,n}\varphi_{r}\rightarrow f$ uniformly as $r\rightarrow 0$.
	\end{itemize}
\end{thm}

The following propositions, involving the spaces $L^p(d\mu_{k, n})$ and $\mathcal{D}_{k,n}(\R)$, are needed for the proof of Theorem \ref{DensityThm} in Subsection \ref{ProofTheorem2}:
\begin{proposition}\label{LpProposition}
	Let \( 1 \leq p < \infty \). For any function \( f \in L^p(d\mu_{k,n}) \), define the sequence of functions \((g_m)_{m\in\mathbb{N}_0}\), such that $g_m\in L^p(d\mu_{k,n})$, by
	\begin{equation}\label{gm}
		g_m(x) = \sum_{\ell=0}^m \binom{m}{\ell}\Bigl(2k+\frac{2}{n}\Bigr)^{m-\ell} \left(x\frac{d}{dx}\right)^{\ell} f(x), \quad x \in \mathbb{R}.
	\end{equation}
	
	Then, the following inequality holds:
	\begin{equation}\label{fLp}
		\|f\|_{L^p(d\mu_{k,n})} \leq \|g_m\|_{L^p(d\mu_{k,n})}.
	\end{equation}
\end{proposition}

\begin{proof}
	Let us denote by \(\mathcal{T}_{k,n}\) the integral operator defined by

	\[
	(\mathcal{T}_{k,n} f)(x) = \int_{0}^{1} f(tx)\, t^{2k+\frac{2}{n}-1}\, dt, \quad x \in \mathbb{R}.
	\]

	For \(t > 0\), the chain rule gives

	\[
	\frac{d}{dt}\Bigl(t^{2k+\frac{2}{n}} f(tx)\Bigr) = \Bigl(2k+\frac{2}{n}\Bigr) t^{2k+\frac{2}{n}-1} f(tx) + t^{2k+\frac{2}{n}-1}\Bigl(tx \frac{d}{d(tx)}f\Bigr)(tx), \quad x \in \mathbb{R}.
	\]

	Integrating the previous expression yields
	\[
	f(x) = \int_0^1 \frac{d}{dt}\Bigl(t^{2k+\frac{2}{n}} f(tx)\Bigr)\, dt, \quad x \in \mathbb{R},
	\]
	which shows that \(\mathcal{T}_{k,n}\) is the inverse of the differential operator
	\[
	x\,\frac{d}{dx} + \Bigl(2k+\frac{2}{n}\Bigr)I.
	\]
	
	Next, let us consider the sequence of functions \(\left(g_m\right)_{m\in\mathbb{N}_0}\) defined by \eqref{gm}. We observe that
	$$	g_m(x) = (\mathcal{T}_{k,n}g_{m+1})(x), \quad x \in \mathbb{R}, \quad m \in \mathbb{N}_0.
	$$
	Furthermore, applying Minkowski's integral inequality, we obtain
	$$
	\|g_m\|_{L^p(d\mu_{k,n})}^p \leq \int_{0}^{1} \left(\int_{\mathbb{R}} \bigl|g_{m+1}(tx)\bigr|^p\, d\mu_{k,n}(x)\right) t^{2k+\frac{2}{n}-1}\, dt, \quad m \in \mathbb{N}_0.
	$$

	Since the measure \(d\mu_{k,n}\) is homogeneous of order \(2k+\frac{2}{n}-1\), i.e.,

	\[
	d\mu_{k,n}(tx) = t^{2k+\frac{2}{n}-1}\, d\mu_{k,n}(x), \quad \forall\, x \in \mathbb{R},\, t > 0,
	\]

	the above inequality simplifies to

	\[
	\|g_m\|_{L^p(d\mu_{k,n})}^p \leq \|g_{m+1}\|_{L^p(d\mu_{k,n})}^p, \quad m \in \mathbb{N}_0.
	\]

	An inductive argument then yields
	\begin{eqnarray*}
		\|f\|_{L^p(d\mu_{k,n})}^p \leq \|g_m\|_{L^p(d\mu_{k,n})}^p, \quad \forall\, m \in \mathbb{N}_0,
	\end{eqnarray*}
	which is equivalent to \eqref{fLp}.
\end{proof}

{\begin{proposition}\label{Lem}
		Let $\varphi\in \mathcal{D}_{k,n}(\R)$ and $f\in C^\infty(\R)$ have bounded supports $\operatorname{supp}(\varphi)\subseteq [-R_1,R_1]$ and $\operatorname{supp}(f)\subseteq [-R_2,R_2]$, where $R_1,R_2>0$. Then:
		\begin{itemize}
			\item[{\bf (a)}] $\operatorname{supp}(f\star_{k,n}\varphi)\subseteq [-R,R]$, with $R=\left((R_1)^\frac{1}{n}+(R_2)^\frac{1}{n}\right)^n$.
			\item[{\bf (b)}] $f\star_{k,n}\varphi$ belongs to $\mathcal{D}_{k,n}(\R)$.
		\end{itemize}
\end{proposition}}

\begin{proof}	
	Starting from \eqref{Kkernel}, the condition $\operatorname{supp}(f) \subseteq [-R_2,R_2]$ implies that $$
	\operatorname{supp}(\tau_{k,n}^yf)\subseteq \left[-\left(|y|^\frac{1}{n}+(R_2)^\frac{1}{n}\right)^n,\left(|y|^\frac{1}{n}+(R_2)^\frac{1}{n}\right)^n\right],$$
	where $\tau_{k,n}^y$ is the generalized translation operator defined in \eqref{translation}. Therefore, the proof of \textbf{(a)} follows immediately from the definition of the convolution operation $\star_{k,n}$, given in equation  \eqref{convolutionIntegral}.
	
	To prove {\bf (b)}, note that $f\star_{k,n}\varphi \in \mathcal{D}_{k,n}(\R)$ if
	$$
	G(f\star_{k,n} \varphi):=\sup_{x\in \R\setminus\{0\}}	\left|\left(~n\left|~x~\right|^{2-\frac{2}{n}}\Delta_k~\right)^\alpha \left( nx\frac{d}{dx}\right)^\ell\left(f\star_{k,n} \varphi\right)(x)\right| 
	$$
	is finite.
	
	By the $\kn$-generalized convolution formula \eqref{convolutionIntegral}, we have
	{$$
		\left(~n\left|~x~\right|^{2-\frac{2}{n}}\Delta_k~\right)^\alpha \left(nx\frac{d}{dx}\right)^\ell \left(f\star_{k,n} \varphi\right)(x)=$$ $$\int_{\mathbb{R}}(\mathcal{F}_{k,n} f)(y)~(\mathcal{F}_{k,n}\varphi)(y)~ \left(~n\left|~x~\right|^{2-\frac{2}{n}}\Delta_k~\right)^\alpha \left(x\frac{d}{dx}\right)^\ell B_{k,n}\left(~(-1)^nx,y~\right)d\mu_{k,n}(y), \quad x\in \mathbb{R}.$$}
	
	Then, by applying Proposition \ref{Bkn-xdxProposition} we obtain 
	\begin{eqnarray*}
		G(f\star_{k,n} \varphi)\leq \int_{\mathbb{R}}\left|~(\mathcal{F}_{k,n} f)(y)\right|~\left|(\mathcal{F}_{k,n}\varphi)(y)~\right|~n^\alpha |y|^{\frac{2\alpha}{n}}{\bf N}_\ell\left(n|xy|^{\frac{1}{n}}\right)d\mu_{k,n}(y),
	\end{eqnarray*}
	where ${\bf N}_\ell(z)$ is the polynomial of degree $\ell$ with all positive coefficients appearing in the estimate~\eqref{Bcondition-ellDeltak}.
	
	So, for every $x\in [-R,R]$ we arrive at the estimate
		\begin{eqnarray}
			\label{ConvIneq}
		G(f\star_{k,n} \varphi)\leq n^\alpha \int_{\mathbb{R}}\left|~(\mathcal{F}_{k,n} f)(y)\right|~|y|^{\frac{2\alpha}{n}}{\bf N}_\ell\left(nR^{\frac{1}{n}}|y|^{\frac{1}{n}}\right)\left|(\mathcal{F}_{k,n}\varphi)(y)~\right|~ d\mu_{k,n}(y).
	\end{eqnarray}
	
	

{Since \(f\in L^1(d\mu_{k,n})\) implies that \(\mathcal{F}_{k,n}f\in L^\infty(d\mu_{k,n})\), and \(\varphi\in \mathcal{D}_{k,n}(\R)\subseteq\mathcal{S}_{k,n}(\R)\)  ensures that  \(\mathcal{F}_{k,n}\varphi\in \mathcal{S}_{k,n}(\R)\) (by Theorem \ref{Fkn-Schwartz-Theorem}),  it follows that the function $$y\mapsto |y|^{\frac{2\alpha}{n}}{\bf N}_\ell\left(nR^{\frac{1}{n}}|y|^{\frac{1}{n}}\right)\mathcal{F}_{k,n}\varphi(y)$$ belongs to $L^1(d\mu_{k,n})$.} 
	Then, from \eqref{ConvIneq} we immediately get
	\begin{eqnarray*}
		G(f\star_{k,n} \varphi)\leq n^\alpha \| \mathcal{F}_{k,n}f\|_{L^\infty(d\mu_{k, n})}~\left\| ~|\cdot|^{\frac{2\alpha}{n}}{\bf N}_\ell\left(nR^{\frac{1}{n}}|\cdot|^{\frac{1}{n}}\right)\mathcal{F}_{k,n}\varphi~\right\|_{L^1(d\mu_{k,n})}<\infty,
	\end{eqnarray*}
as required.
\end{proof}

\subsection{Proof of Theorem \ref{DensityThm} }\label{ProofTheorem3}

\begin{proof}[Proof of {\bf (i)}]
Let \(f \in \mathcal{S}_{k,n}(\mathbb{R})\) and $(g_m)_{m\in \N_0}$ be the sequence of functions defined by \eqref{gm} in Proposition \ref{LpProposition}. From the observation that 
\begin{equation}\label{Qgm}\Bigl(1+|x|^{\frac{2}{n}}\Bigr)^\beta g_m(x) = \sum_{\alpha=0}^{\beta} \binom{\beta}{\alpha} \Bigl(|x|^{\frac{2}{n}}\Bigr)^\alpha g_m(x), \quad x \in \mathbb{R},
\end{equation}
 one can easily check that for every \(\beta, m\in \mathbb{N}_0\), $g_m\in \mathcal{S}_{k,n}(\R)$, the quantity \[ Q_{\beta}(g_m):=\sup_{x\in \R\setminus \{0\}}	\Bigl(1+|x|^{\frac{2}{n}}\Bigr)^\beta \bigl|g_m(x)\bigr| \] is finite.

Then, by choosing $\beta\in \N$ such that \(\beta > \frac{n}{2p}\Bigl(2k+\frac{2}{n}-1\Bigr)\), it holds from \eqref{knMeasure} that \[ \left\|\Bigl(1+|\cdot|^{\frac{2}{n}}\Bigr)^{-\beta}\right\|_{L^p(d\mu_{k,n})}^p = c_{k,n} \int_{\mathbb{R}} \Bigl(1+|x|^{\frac{2}{n}}\Bigr)^{-\beta p}|x|^{2k+\frac{2}{n}-2}\, dx \] is convergent.

Thus, from the inequality \eqref{fLp} given by Proposition \ref{LpProposition} and from \begin{eqnarray*}
	\|g_m\|_{L^p(d\mu_{k,n})} &=& \left(\int_{\mathbb{R}} \Bigl(1+|x|^{\frac{2}{n}}\Bigr)^{-\beta p} \Bigl|\Bigl(1+|x|^{\frac{2}{n}}\Bigr)^{\beta}g_m(x)\Bigr|^p\, d\mu_{k,n}(x)\right)^{\frac{1}{p}}
\end{eqnarray*} we get \begin{eqnarray*}
	\|f\|_{L^p(d\mu_{k,n})}  \leq 	\|g_m\|_{L^p(d\mu_{k,n})} \leq   \left\|\Bigl(1+|\cdot|^{\frac{2}{n}}\Bigr)^{-\beta}\right\|_{L^p(d\mu_{k,n})} Q_\beta(g_m).
\end{eqnarray*}

{Using \eqref{Qgm}, \eqref{gm} and \eqref{NormalOrdering}, we get 
$$Q_\beta(g_m)\leq \sum_{\alpha=0}^{\beta} \binom{\beta}{\alpha} \sum_{\ell=0}^m \binom{m}{\ell}\Bigl(2k+\frac{2}{n}\Bigr)^{m-\ell} \sum_{j=1}^{\ell} S(\ell,j)~ P_{\alpha,0}( X^jD^jf).$$ 

Since $P_{\alpha,\beta}(X^j D^j f)$ define the seminorms on $\mathcal{S}_{k,n}(\R)$, then we deduce the continuous embedding $\mathcal{S}_{k,n}(\R)\hookrightarrow L^p(d\mu_{k,n})$.} 
\end{proof}

\begin{proof}[Proof of {\bf (ii)}]
		Since statement {\bf (i)} and Proposition \ref{Dkn-thm} (see statement {\bf (ii)}) already yield the continuous embeddings 
	\[
	\mathcal{S}_{k,n}(\mathbb{R}) \hookrightarrow L^p(d\mu_{k,n})  \quad \text{and} \quad \mathcal{D}_{k,n}(\mathbb{R}) \hookrightarrow \mathcal{S}_{k,n}(\mathbb{R}),
	\]
	it remains only to show that \(\mathcal{D}_{k,n}(\mathbb{R})\) is dense in \(L^p(d\mu_{k,n})\).
	
Let \(f \in L^p\bigl(d\mu_{k,n}\bigr)\). Since \(\mathcal{D}(\mathbb{R})\) (i.e., the space of \(C^\infty\) functions on \(\mathbb{R}\) with compact support) is dense in \(L^p\bigl(d\mu_{k,n}\bigr)\), there exists a function \(g \in \mathcal{D}(\mathbb{R})\) such that for every \(\varepsilon > 0\), there holds
	\begin{equation*}
		\|f-g\|_{L^p\bigl(d\mu_{k,n}\bigr)} < \frac{\varepsilon}{2}.
	\end{equation*}
	
Choose \(\varphi \in \mathcal{D}_{k,n}(\mathbb{R})\) satisfying

	\[
	\int_{\mathbb{R}} \varphi(x)\, d\mu_{k,n}(x) = 1,
	\]
	and consider the family of functions \((\varphi_r)_{r>0}\) defined as in \eqref{rDilation}.
By Definition \ref{Dkn-definition}, one has that  \((\varphi_r)_{r>0}\) forms an approximation of the identity such that $\varphi_r\in \mathcal{D}_{k,n}(\R)$, for each $r>0$. Consequently, by Theorem \ref{BeNe2-Thm} the convolution \(\varphi_r \star_{k,n} g\) converges to \(g\) in \(L^p(d\mu_{k,n})\). That is,
	\begin{equation*}
		\forall\,\varepsilon>0~\exists\, r_0>0 \text{~:~} \forall\,r\geq r_0 \Longrightarrow  \|{\varphi_{r}}\star_{k,n} g - g\|_{L^p\bigl(d\mu_{k,n}\bigr)} < \frac{\varepsilon}{2}.
	\end{equation*}
	
Moreover, Lemma \ref{Lem} ensures that \({\varphi_{r}}\star_{k,n} g \in \mathcal{D}_{k,n}(\mathbb{R})\). Then, for any \(r \ge r_0\), we obtain by the triangle inequality
	\begin{align*}
		\|f - {\varphi_{r}}\star_{k,n} g\|_{L^p\bigl(d\mu_{k,n}\bigr)} &\leq \|f-g\|_{L^p\bigl(d\mu_{k,n}\bigr)} + \|{\varphi_{r}}\star_{k,n} g - g\|_{L^p\bigl(d\mu_{k,n}\bigr)} \\
		&< \frac{\varepsilon}{2} + \frac{\varepsilon}{2} \\
		&= \varepsilon.
	\end{align*}

So, \(\mathcal{D}_{k,n}(\mathbb{R})\) is dense in \(L^p\bigl(d\mu_{k,n}\bigr)\), which completes the proof.
\end{proof}

\section{Conclusion and Open Problems}\label{ConclusionOpen}

\subsection{Conclusion}

We introduced a Schwartz-type space $\mathcal{S}_{k,n}(\R)$, invariant under the one-dimensional $\kn-$generalized Fourier transform $\mathcal{F}_{k,n}:=\mathcal{F}_{k,\frac{2}{n}}$. Its continuous and dense embedding into the weighted Lebesgue spaces $L^p(d\mu_{k,n})$ for every $1 \leq p < \infty$ makes them the appropriate function space for the $\kn-$generalized Fourier transform. In particular, Theorems~\ref{Fkn-Schwartz-Theorem} -- \ref{DensityThm} fail in general for the classical Schwartz space $\mathcal{S}(\mathbb{R}^{N})$ when $a \neq 2$, as proved in \cite[Example 5.1]{GoIvTi23}.

Our approach demonstrates that one can restore invariance by modifying the seminorms in a way compatible with the Laguerre semigroup of Ben Sa\"id--Kobayashi--\O rsted \cite{BKO} for $N=1$ and $\displaystyle a=\frac{2}{n}$. The analytic tools are closely related to R\"osler's work on product formulas and translations \cite{Ro99,Ro03}, and involve product formulas for Bessel-type kernels \cite{BNS}, and approximation-of-identity type  arguments~\cite{BeNe2}. This fuses harmonic analysis with representation theory, extending to rapidly decaying functions -- such as the deformations of the Gaussian $e^{-x^2}$, given by the sequence $\left(f_n\right)_{n\in \N}$ defined by $\displaystyle f_n(x)=e^{-n|x|^{\frac{2}{n}}}$. Figures \ref{fig1}--\ref{fig3} numerically illustrate these functions and their $\mathcal{F}_{k,n}$ images for $n=1,\ldots,5$,  confirming their rapidly decreasing decay at infinity (and hence belongs to every $L^p$-space) and the shape preservation of $(\mathcal{F}_{k,n} f_n)(y)$ with seminorm decay (unlike the distorted case for $n=1$).


First, we notice that these functions and their $\kn-$Fourier transforms decay exponentially at infinity, landing them comfortably in every $L^p-$space.
Second, each $f(x)=e^{-n|x|^{\frac{2}{n}}}$ mirrors the shape of its own $(\mathcal{F}_{k,n} f)(y)$ -- unlike the Gaussian ($n=1$), whose $\kn-$generalized Fourier distorts its smooth shape (see Figure \ref{fig3}). 
These observations support the choice of $\mathcal{S}_{k,n}(\mathbb{R})$ rather than the classical Schwartz space $\mathcal{S}(\mathbb{R})$ to deal with $\kn-$generalized Fourier transforms in the one-dimensional case.

In conclusion, our approach is a first step toward obtaining Schwartz-type spaces adapted to $\kn-$generalized Fourier transforms.
This suggests that a similar strategy could be employed in higher dimensions, which is the topic we will now address.

\begin{figure}[ht]
\begin{tikzpicture}
	\begin{axis}[
		width=12cm, height=8cm,
		xlabel={$x$},
		ylabel={$f(x)$},
		xmin=-3, xmax=3,
		ymin=0, ymax=1.1,
		grid=both,
		legend style={at={(0.98,0.98)}, anchor=north east, cells={align=left}},
		legend cell align=left,
		smooth,
		thick,
		domain=-3:3,
		samples=200,
		]
		
		\addplot[blue, line width=1.5pt] {exp(-x^2)};
		\addlegendentry{$n=1$ };
		
		\addplot[red, dashed, line width=1.5pt] {exp(-2*abs(x)^(1))}; 
		\addlegendentry{$n=2$};
		
		\addplot[green!70!black, dashed, line width=1.5pt] {exp(-3*abs(x)^(2/3))};
		\addlegendentry{$n=3$};
		
		\addplot[orange, dashed, line width=1.5pt] {exp(-4*abs(x)^(0.5))}; 
		\addlegendentry{$n=4$};
		
		\addplot[purple, dashed, line width=1.5pt] {exp(-5*abs(x)^(0.4))}; 
		\addlegendentry{$n=5$};
		
	\end{axis}
\end{tikzpicture}
	\caption{Plot of $f(x)= e^{-n|x|^{\frac{2}{n}}}$ for $n\in \{1,2,3,4,5\}$.}
\label{fig1}
\end{figure}

\begin{figure}[ht]
\begin{tikzpicture}
	\begin{axis}[
		width=12cm, height=8cm,
		xlabel={$y$},
		ylabel={$(\mathcal{F}_{1,n} f)(y)$},
		xmin=-8, xmax=8,
		ymin=0, ymax=0.4,
		grid=both,
		legend style={at={(0.98,0.98)}, anchor=north east, cells={align=left}},
		legend cell align=left,
		smooth,
		thick,
		domain=-8:8,
		samples=300,
		]
		
		
		\addplot[blue, line width=1.5pt] {1/(2^(1 - 0.5 + 1)) * exp(-1*abs(x)^(2/1)/4)};
		\addlegendentry{$n=1$};
		
		\addplot[red, dashed, line width=1.5pt] {1/(2^(2 - 1 + 1)) * exp(-2*abs(x)^(1)/4)};
		\addlegendentry{$n=2$};
		
		\addplot[green!70!black, dashed, line width=1.5pt] {1/(2^(3 - 1.5 + 1)) * exp(-3*abs(x)^(2/3)/4)};
		\addlegendentry{$n=3$};
		
		\addplot[orange, dashed, line width=1.5pt] {1/(2^(4 - 2 + 1)) * exp(-4*abs(x)^(0.5)/4)};
		\addlegendentry{$n=4$};
		
		\addplot[purple, dashed, line width=1.5pt] {1/(2^(5 - 2.5 + 1)) * exp(-5*abs(x)^(0.4)/4)};
		\addlegendentry{$n=5$};
		
	\end{axis}
\end{tikzpicture}
	\caption{Plot of the $(k,1)-$generalized Fourier transform of $f(x)=e^{-n|x|^{\frac{2}{n}}}$ for $n\in \{1,2,3,4,5\}$.}
\label{fig2} 
\end{figure}

\begin{figure}[ht]
	\centering
	\begin{tikzpicture}
		\begin{axis}[
		width=12cm, height=8cm,
		xlabel={$y$},
		ylabel={$(\mathcal{F}_{k,n} f)(y)$},
		xmin=-3, xmax=3,
		ymin=0, ymax=2.1,
		grid=both,
		legend style={at={(0.98,0.98)}, anchor=north east, cells={align=left}},
		legend cell align=left,
		smooth,
		thick,
		domain=-3:3,
		samples=200,
		]
			
			\addplot[blue, line width=1.5pt] table [x index=0, y index=1, col sep=comma] {F_a_2.000.csv};
			\addlegendentry{$n=1$}
			
			\addplot[red, dashed, line width=1.5pt] table [x index=0, y index=1, col sep=comma] {F_a_1.000.csv};
			\addlegendentry{$n=2$}
			
			\addplot[green!70!black, dashed, line width=1.5pt] table [x index=0, y index=1, col sep=comma] {F_a_0.667.csv};
			\addlegendentry{$n=3$}
			
			\addplot[orange, dashed, line width=1.5pt]  table [x index=0, y index=1, col sep=comma] {F_a_0.500.csv};
			\addlegendentry{$n=4$}
			
			\addplot[purple, dashed, line width=1.5pt]  table [x index=0, y index=1, col sep=comma] {F_a_0.400.csv};
			\addlegendentry{$n=5$}
			
		\end{axis}
	\end{tikzpicture}
\caption{Plot of the $\kn-$generalized Fourier transform of $f(x)=e^{-x^{2}}$ for $n\in \{1,2,3,4,5\}$.}
	\label{fig3}
\end{figure}




\subsection{Open Problems}

The construction of $\mathcal{S}_{k,n}(\mathbb{R})$  is rooted in  $\mathfrak{sl}(2,\R)–$symmetries and formally extends to an arbitrary dimension $N$ and all parameters $a>0$. Extending our results beyond $N=1$ requires uniform bounds or precise polynomial growth estimates for the kernel $B_{k,a}$ in higher dimensions, which are currently known only in special cases such as $a=1,2$ or $k\equiv 0$ with $a\in \mathbb{Q}_+$. Guided by the known estimates for $B_{k,a}$ when $a=1,2$ and by the recent work of De Bie, Lian and Maes \cite{deBieLM25} on kernel bounds in the case $k\equiv 0$ and $a\in \mathbb{Q}_+$, we formulate the following conjecture:

\newpage 
~
\newpage 

\begin{conjecture}\label{kaSchwartz_Conjecture}Let $\mathcal{S}_{k,a}(\R^N)$ be the space of all {$f\in C^\infty(\R^N\setminus\{0\})$}	such that 
	\begin{eqnarray*} \sup_{x\in \R^N\setminus\{0\}}~\left|~\left(\left\|x\right\|^{a}\right)^{\alpha} \left(\left\|x\right\|^{2-a}\Delta_{k}\right)^\beta \left(\|x\|^\ell ~\frac{\partial^\ell f(x)}{\partial \|x\|^\ell}\right)~\right|<\infty, ~~\mbox{for all}~~\alpha,\beta,\ell \in \N_0 \end{eqnarray*}
	
 Then, 	for the $(k,a)-$generalized Fourier transform as defined in \eqref{FkaRN} we have that $$\mathcal{F}_{k,a}\left(\mathcal{S}_{k,a}(\R^N)\right)=\mathcal{S}_{k,a}(\R^N),$$ provided that \(a>0\) and \(N \in \mathbb{N}\) are such that\begin{enumerate}
\item $a=1$ or $a=2$.
\item $a=\dfrac{2^m}{n}$ ($m,n\in \N$) and $N=2$.
\item[~]
\item $a=\dfrac{p}{q}\in \mathbb{Q}_+$ and $N\geq 2$ is even.	\end{enumerate}\end{conjecture}

Case (2) of Conjecture~\ref{kaSchwartz_Conjecture} is consistent with the uniform boundedness of $B_{k,a}$ when $a=1,2$ (cf.~\cite[Theorem~5.11]{BKO}) and of $B_{0,a}$ for $N=2$ and $\displaystyle a=\frac{2^m}{n}$ (cf.~\cite[Theorem~5.4]{deBieLM25}). It is also compatible with polynomial growth estimates for even $\displaystyle N\ge 2$ (cf.~\cite[Theorem~3.3]{deBieLM25}) and with the failure of boundedness in certain cases such as $a=4$ (cf.~\cite[Theorem~2.1]{deBieLM25}). 

For odd dimensions $N>1$, Conjecture \ref{kaSchwartz_Conjecture} remains open and seems to require a finer large-parameter analysis of the kernel $B _{k,a}$. Such estimates would also clarify whether the even–odd dichotomy, which is already apparent for $k\equiv 0$, persists in the $(k,a)-$generalized Fourier setting. Specifically, one would need estimates that are uniform in the angular variables and compatible with the commutation relations of the $\mathfrak{sl}(2,\R)$–action, to reproduce the $L^{1}-L^{\infty}$ kernel estimates and convolution bounds used in the one-dimensional proof. 

Making progress in this direction would clarify how the $(k,a)-$generalized Fourier transform behaves differently in even and odd dimensions, and whether the results already proved in \cite{deBieLM25} for $k \equiv 0$ also holds for a general $k$.

Beyond Conjecture \ref{kaSchwartz_Conjecture}, several concrete problems appear within reach, even in low-dimensional test cases such as $k\equiv 0$, $\displaystyle a=\frac{2^m}{n}$. Natural questions include establishing continuous embeddings and density results for $\mathcal{S}_{k,a}(\R^{N})$ in suitable weighted $L^{p}$–spaces whenever $\mathcal{F}_{k,a}-$invariance is available, extending Theorem~\ref{DensityThm} to higher dimensions, as well as formulating and proving uncertainty principles, Paley–Wiener type theorems, and Hardy-type inequalities for $\mathcal{F}_{k,a}$ within the framework of $\mathcal{S}_{k,a}(\mathbb{R}^{N})$, in the spirit of \cite{Jo}. 

\end{document}